\documentclass[letterpaper,10pt,reqno]{amsart}
\usepackage{indentfirst} % Indent first paragraph after section header
\usepackage{amssymb}
\usepackage{mathtools} % improves amsmath
\usepackage{mathabx} % makes many symbols (as $\leq$) more beautiful
\usepackage{amsthm}
\usepackage{thmtools}
\usepackage{bbm}       % to use $\mathbbm{1}$
\usepackage{enumitem} % special itemize with custom tags and labels that work
\usepackage[colorlinks=true]{hyperref} % Links in the pdf. Backref option: switched off
\usepackage[usenames,dvipsnames]{xcolor} % With XCOLOR, printed quality is good, (but with COLOR it's bad)
\usepackage{ifthen}
\usepackage{tikz}
\usetikzlibrary{decorations.pathreplacing}
\usepackage{tikz-cd}
\makeatletter

\def\MRbibitem{\@ifnextchar[\my@lbibitem\my@bibitem}

\def\mybiblabel#1#2{\@biblabel{{\hyperref{http://www.ams.org/mathscinet-getitem?mr=#1}{}{}{#2}}}}

\def\myhyperanchor#1{\Hy@raisedlink{\hyper@anchorstart{cite.#1}\hyper@anchorend}}

\def\my@lbibitem[#1]#2#3#4\par{%
  \item[\mybiblabel{#2}{#1}\myhyperanchor{#3}\hfill]#4%
  \@ifundefined{ifbackrefparscan}{}{\BR@backref{#3}}%
  \if@filesw{\let\protect\noexpand\immediate% write to aux-file
    \write\@auxout{\string\bibcite{#3}{#1}}}\fi\ignorespaces%
}

\def\my@bibitem#1#2#3\par{%
  \refstepcounter\@listctr% standard tex item counter for the generic item number
  \item[\mybiblabel{#1}{\the\value\@listctr}\myhyperanchor{#2}\hfill]#3%
  \@ifundefined{ifbackrefparscan}{}{\BR@backref{#2}}%
  \if@filesw\immediate\write\@auxout% write to aux-file
    {\string\bibcite{#2}{\the\value\@listctr}}\fi\ignorespaces%
}

\makeatother
\DeclareFontFamily{U} {MnSymbolA}{}
\DeclareFontShape{U}{MnSymbolA}{m}{n}{
   <-6> MnSymbolA5
   <6-7> MnSymbolA6
   <7-8> MnSymbolA7
   <8-9> MnSymbolA8
   <9-10> MnSymbolA9
   <10-12> MnSymbolA10
   <12-> MnSymbolA12}{}
\DeclareFontShape{U}{MnSymbolA}{b}{n}{
   <-6> MnSymbolA-Bold5
   <6-7> MnSymbolA-Bold6
   <7-8> MnSymbolA-Bold7
   <8-9> MnSymbolA-Bold8
   <9-10> MnSymbolA-Bold9
   <10-12> MnSymbolA-Bold10
   <12-> MnSymbolA-Bold12}{}
\DeclareSymbolFont{MnSyA} {U} {MnSymbolA}{m}{n}
 \DeclareFontFamily{U} {MnSymbolC}{}
\DeclareFontShape{U}{MnSymbolC}{m}{n}{
  <-6> MnSymbolC5
  <6-7> MnSymbolC6
  <7-8> MnSymbolC7
  <8-9> MnSymbolC8
  <9-10> MnSymbolC9
  <10-12> MnSymbolC10
  <12-> MnSymbolC12}{}
\DeclareFontShape{U}{MnSymbolC}{b}{n}{
  <-6> MnSymbolC-Bold5
  <6-7> MnSymbolC-Bold6
  <7-8> MnSymbolC-Bold7
  <8-9> MnSymbolC-Bold8
  <9-10> MnSymbolC-Bold9
  <10-12> MnSymbolC-Bold10
  <12-> MnSymbolC-Bold12}{}
\DeclareSymbolFont{MnSyC} {U} {MnSymbolC}{m}{n}

\DeclareMathSymbol{\top}{\mathord}{MnSyA}{219} % smaller symbol for transpose
\DeclareMathSymbol{\plus}{\mathord}{MnSyC}{20} % a smaller plus sign

\declaretheorem[numberwithin=section]{theorem}

\declaretheorem[sibling=theorem]{proposition}
\declaretheorem[sibling=theorem,style=definition]{definition}

\declaretheorem[sibling=theorem,style=remark]{remark}

\hypersetup{bookmarksdepth = 3} % Depth of sections/subs... to have bookmark links in the pdf;
\numberwithin{equation}{section}     % Makes labeled equations easier to find.

\setlist[enumerate,1]{label={\upshape(\alph*)},ref=\alph*}
\setlist[enumerate,2]{label={\upshape(\arabic*)},ref=\arabic*}

\newcommand{\R}{\mathbb{R}}

\newcommand{\N}{\mathbb{N}}
\newcommand{\E}{\mathbb{E}}

\def\eps{\varepsilon}
\def\phi{\varphi}
\def\R{{\mathbb R}}

\def\N{{\mathbb N}}

\def\cont{{\mathcal C}}

\def\le{\leqslant}
\def\ge{\geqslant}

\def\m{\mathbb{M}}

\DeclareMathOperator{\inte}{int}

\newcommand{\vertiii}[1]{{\left\vert\kern-0.25ex\left\vert\kern-0.25ex\left\vert #1 
    \right\vert\kern-0.25ex\right\vert\kern-0.25ex\right\vert}}
\newcommand{\invertiii}[1]{{\vert\kern-0.25ex\vert\kern-0.25ex\vert #1 
    \vert\kern-0.25ex\vert\kern-0.25ex\vert}}

\begin{document}

\title{Orbits closeness for slowly mixing dynamical systems}
\date{\today}
\author[J. Rousseau]{J\'er\^ome Rousseau}
\address{J\'er\^ome Rousseau, CREC, Acad\'emie Militaire de St Cyr Co\"etquidan, 56381 GUER Cedex, France}
\address{IRMAR, CNRS UMR 6625,
Universit\'e de Rennes 1, 35042 Rennes, France}
\address{
Departamento de Matem\'atica\\
Universidade Federal da Bahia\\
Avenida Ademar de Barros s/n\\
40170-110 Salvador, BA\\
Brasil} 
\email{\href{mailto:jerome.rousseau@ufba.br}{jerome.rousseau@ufba.br}}
\urladdr{\url{http://www.sd.mat.ufba.br/~jerome.rousseau/}}

\author[M. Todd]{Mike Todd}
\address{Mike Todd\\ Mathematical Institute\\
University of St Andrews\\
North Haugh\\
St Andrews\\
KY16 9SS\\
Scotland} \email{\href{mailto:m.todd@st-andrews.ac.uk}{m.todd@st-andrews.ac.uk}}
\urladdr{\url{https://mtoddm.github.io/}}


\thanks{We would like to thank Thomas Jordan for helpful suggestions on correlation dimension.  We also thank the referee(s) for useful comments. Both authors were partially supported by FCT projects PTDC/MAT-PUR/28177/2017 and by CMUP (UIDB/00144/2020), which is funded by FCT with national (MCTES) and European structural funds through the programs FEDER, under the partnership agreement PT2020.  JR was also partially supported by CNPq and PTDC/MAT-PUR/4048/2021, and with national funds.}

\begin{abstract}
Given a dynamical system, we prove that the shortest distance between two $n$-orbits scales like $n$ to a power even when the system has slow mixing properties, thus building and improving on results of Barros, Liao and the first author.  We also extend these results to flows.  Finally, we give an example for which the shortest distance between two orbits has no scaling limit.
\end{abstract}

\maketitle

\section{Introduction}

The study of the statistical properties of dynamical systems is one of the main pillars of ergodic theory. In particular, one of the principal lines of investigation is to try and obtain quantitative information on the long term behaviour of orbits (such as return and hitting times, dynamical extremal indices or logarithm laws). 

In a metric space $(X,d)$, the problem of the shortest distance between two orbits of a dynamical system $T:X\to X$, with an ergodic measure $\mu$, was introduced in \cite{BarLiaRou19}. That is, for $n\in \N$ and $x, y\in X$, they studied
$$\m_n(x, y)= \m_{T, n}(x, y):=\min_{0\le i,j\le n-1} d(T^i(x), T^j(y))$$
and showed that the decay of $\m_n$ depends on the correlation dimension.

The \emph{lower correlation dimension} of $\mu$ is defined by 
$$\underline{C}_\mu:=\liminf_{r\to 0} \frac{\log\int\mu(B(x, r))~d\mu(x)}{\log r},$$
and the \emph{upper correlation dimension} $\overline{C}_\mu$ is analogously defined via the limsup.  If these are equal, then this is $C_\mu$, the \emph{correlation dimension} of $\mu$. This dimension plays an important role in the description of the fractal structure of invariant sets in dynamical systems and has been widely studied from different points of view: numerical estimates (e.g. \cite{badii, vaienti-gen, SR}), existence and relations with other fractal dimension (e.g. \cite{barbaroux, Pes93}) and relations with other dynamical quantities (e.g. \cite{faranda-vaienti, mantica}).

It is worth mentioning that the problem of the shortest distance between orbits is a generalisation of the longest common substring problem for random sequences, a key feature in bioinformatics and computer science (see e.g \cite{waterman}).

In \cite[Theorem 1]{BarLiaRou19}, under the assumption $\underline{C}_\mu>0$, a general lower bound for $\m_n$ was obtained:
\begin{theorem}For a dynamical system $(X,T,\mu)$, we have
\begin{equation*}
\limsup_{n}\frac{\log \m_{T, n}(x, y)}{-\log n} \le \frac{2}{\underline{C}_\mu} \qquad \mu\times\mu\text{-a.e. } x, y.
\end{equation*}
\label{thm:BLR}
\end{theorem}

To replace the inequality above with equality, in  \cite[Theorems 3 and 6]{BarLiaRou19}, the authors assumed that $C_\mu$ exists and proved
\begin{equation}
\liminf_{n}\frac{\log \m_{T,n}(x, y)}{-\log n} \ge \frac{2}{C_\mu} \qquad \mu\times\mu\text{-a.e. } x, y,\label{eq:WTS}
\end{equation}
using some exponential mixing conditions on the system. 

One could naturally wonder if this mixing condition could be relaxed or even dropped. In \cite{BarLiaRou19}, a partial answer was given and it was proved that for irrational rotation (which are not mixing) the inequality in Theorem~\ref{thm:BLR} could be strict.

In this paper, we extend the above results \eqref{eq:WTS} to discrete systems with no requirement on mixing conditions. The main tool in proving our positive results is inducing: the idea in the discrete case is to first take advantage of the fact that Theorem~\ref{thm:BLR} holds in great generality (including, as we note later, to higher-dimensional hyperbolic cases) and then to show that if there is an induced version of the system satisfying \eqref{eq:WTS} then this inequality passes to the original system. 

 Moreover, we also extend the results of \cite{BarLiaRou19} to flows. Thus we first have to prove an analogue of Theorem~\ref{thm:BLR} and observe that in the continuous setting, the correct scaling is $C_\mu-1$. Then, using inducing via Poincar\'e sections, we also obtain an analogue of \eqref{eq:WTS}.
 
 We will give examples for all of these results both in the discrete and continuous setting. We also give a class of examples in Section~\ref{sec:skew} where the conclusions of \cite{BarLiaRou19} fail to hold.  This class is slowly mixing and also does not admit an induced version

 Finally, we emphasise that one of the obstacles to even wider application is proving that the correlation dimension $C_\mu$ exists, see Section~\ref{ssec:cordim} for some discussion and results. For suspension flows, under some natural assumptions, we will show in Section~\ref{sec:suspflow}, that if the correlation dimension of the base exists then the correlation dimension of the invariant measure of the flow also exists.

\section{Main results and proofs for orbits closeness in the discrete case}

\subsection{The main theorem in the nonuniformly expanding case}

We will suppose that given $(X, T, \mu)$, there is a subset $Y=\overline{\bigcup_iY_i}\subset X$ and an \emph{inducing time} $\tau:Y\to \N\cup\{\infty\}$, constant on each $Y_i$, and denoted $\tau_i$ and so that our induced map is $F=T^\tau: Y\to Y$.  We suppose that there is an $F$-invariant probability measure $\mu_F$ with $\int\tau~d\mu_F<\infty$ and which  \emph{projects} to $\mu$ by the following rule:
\begin{equation}\mu(A)=\frac{1}{\int\tau~d\mu_F} \sum_i\sum_{k=0}^{\tau_i-1}\mu_F(Y_i\cap T^{-k}(A)).
\label{eq:indmeas}
\end{equation}
We call $(Y, F, \mu_F)$ an \emph{inducing scheme}, or an \emph{induced system}  for $(X, T, \mu)$.  For systems which admit an inducing scheme, we have our main theorem:

\begin{theorem}Assume that the inducing scheme $(Y, F, \mu_F)$ satisfies \eqref{eq:WTS} and that $C_{\mu_F}=C_{\mu}$. Then
$$
\lim_{n\to \infty}\frac{\log \m_{T, n}(x, y)}{-\log n} = \frac{2}{C_{\mu}} \qquad \mu\times\mu\text{-a.e. } x, y.$$
\label{thm:induced_exp}
\end{theorem}

In Section~\ref{sec:skew} we give an example of a class of mixing systems where the conclusion of this theorem fails.  These systems do not have good inducing schemes, see Remark~\ref{rmk:noind} below.

\begin{remark}
As can be seen from the proof of this theorem, as well as related results in this paper, in fact what we prove is that if there is an induced system satisfying 
$$\liminf_{n}\frac{\log \m_{F,n}(x, y)}{-\log n} \ge \frac{2}{\overline{C}_{\mu_F}} \qquad \mu_F\times\mu_F\text{-a.e. } x, y,$$
then 
$$\liminf_{n\to \infty}\frac{\log \m_{T, n}(x, y)}{-\log n} \ge \frac{2}{\overline{C}_{\mu_F}} \qquad \mu\times\mu\text{-a.e. } x, y,$$
with the analogous statements for flows in Section~\ref{sec:flows}.
\end{remark}

In Section~\ref{sec:disegs}  we will give examples of systems where $\{Y_i\}_i$ is countable, $\mu$ and $\mu_F$ are absolutely continuous with respect to Lebesgue and $C_{\mu_F}=C_{\mu}$.

\begin{proof}[Proof of Theorem~\ref{thm:induced_exp}]
The main observation here is that it is sufficient to prove that $\lim_{k\to \infty} \frac{\log \m_{T,n_k}\left(x, y\right)}{-\log n_k} \ge   \frac{2}{C_{\mu}}$ along a subsequence $(n_k)_k$ which scales linearly with $k$.

For $x\in Y$, define $\tau_n(x):=\sum_{k=0}^{n-1} \tau(F^k(x))$. 
Given $\eps>0$ and $N\in \N$, set
$$U_{\eps, N}:=\left\{x\in Y:|\tau_n(x)-n\bar\tau|\le \eps n\text{ for all } n\ge N\right\}.$$
These are nested sets and by Birkhoff's Ergodic Theorem, we have $\lim_{N\to \infty}\mu_F(U_{\eps, N})=1$.  In particular, by \eqref{eq:indmeas}, $\mu(U_{\eps, N})>0$ for $N$ sufficiently large and hence $$\lim_{n\to \infty}\mu\left(\bigcup_{i=0}^{\lfloor \eps N\rfloor}T^{-i}(U_{\eps, N})\right)=1.$$
So for $\mu\times\mu$-typical $(x, y)\in X\times X$, there is $N\in \N$ such that $x, y\in  \bigcup_{i=0}^{\lfloor \eps N\rfloor}T^{-i}(U_{\eps, N})$.  Set $i, j\le \lfloor \eps N\rfloor$ minimal such that $T^i(x), T^{j}(y)\in U_{\eps, N}\subset Y$.  Then   \cite[Theorem 3]{BarLiaRou19} implies that for any $\eta>0$ and sufficiently large $n$,
$$\frac{\log \m_{F,n}(T^i(x), T^{j}(y))}{-\log n} \ge \frac{2}{C_{\mu_F}}-\eta=\frac{2}{C_{\mu}}-\eta.$$
Putting together the facts that the $n$-orbit by $F$ of $T^i(x)$ (respectively $T^{j}(y)$) is a subset of the $\tau_n(x)$- (respectively $\tau_n(y)$-) orbit by $T$ of $T^i(x)$ (respectively $T^{j}(y)$) and that $i, j, |\tau_n(T^i(x))-n\bar\tau|, |\tau_n(T^{j}(y))-n\bar\tau| \le n\eps$ for $n\ge N$, we obtain
$$\m_{T,n\lceil \bar\tau+2\eps\rceil}\left(x, y\right)\leq  \m_{T,n\lceil \bar\tau+\eps\rceil}\left(T^i(x), T^{j}(y)\right)\leq \m_{F,n}(T^i(x), T^{j}(y))$$
and thus
$$ \frac{\log \m_{T,n\lceil \bar\tau+2\eps\rceil}\left(x, y\right)}{-\log n} \ge \frac{2}{C_{\mu}}-\eta.$$
Observing that $\lim_{n\to \infty}\frac{\log n\lceil \bar\tau+2\eps\rceil}{\log n}=1$ and taking limit in the previous equation we deduce that $\lim_{n\to \infty}\frac{\log \m_{T,n}\left(x, y\right)}{-\log n} \ge   \frac{2}{C_{\mu}}-\eta$.  Since $\eta$ can be choose arbitrary small, the theorem is proved.
\end{proof}

\subsection{The main theorem in the nonuniformly hyperbolic case}

We next consider systems $T:X\to X$ with invariant measure $\mu$ which are nonuniformly hyperbolic in the sense of Young, see \cite{You98}.  Then there is some $Y\subset X$ and an inducing time $\tau$ defining $F=T^\tau:Y\to Y$, with measure $\mu_F$, which is uniformly expanding modulo uniformly contracting directions.  We can quotient out these contracting directions to obtain a system $\bar F: \bar Y\to \bar Y$, which has an invariant measure $\mu_{\bar F}$.

\begin{theorem}
Assume that the induced system $(\bar Y, \bar F, \mu_{\bar F})$ satisfies \eqref{eq:WTS} and that $C_{\mu}=C_{\mu_F}$. Then
$$
\lim_{n\to \infty}\frac{\log \m_{T, n}(x, y)}{-\log n} = \frac{2}{C_{\mu}} \qquad \mu\times\mu\text{-a.e. } x, y.$$
\label{thm:induced_hyp}
\end{theorem}

The proof is directly analogous to that of Theorem~\ref{thm:induced_exp}.

\subsection{Requirements on the induced system}

In  \cite{BarLiaRou19}, the main requirement for \eqref{eq:WTS} to hold is that the system has some Banach space $\cont$ of functions from $X$ to $\R$, $\theta\in (0, 1)$ and $C_1\geq0$ such that for all $\phi, \psi\in \cont$ and $n\in\N$,
$$\left|\int \psi\cdot\phi\circ F^n~d\mu_F-\int\psi~d\mu_F\int\phi~d\mu_F\right|\le C_1\|\phi\|_{\cont} \|\psi\|_{\cont}\theta^n.$$
Some regularity conditions on the norms of characteristics on balls and the measures were also required, as well as a topological condition on our metric space (always satisfied for subset of $\R^n$ with the Euclidean metric and subset of a Riemannian manifold of bounded curvature), but we leave the details to  \cite{BarLiaRou19}. We can also remark that for Lipschitz maps on a compact metric space with $\cont = Lip$, these regularity conditions can be dropped \cite{GouRStad}.

 In \cite[Theorem 3]{BarLiaRou19} the main application was to systems where $\cont = BV$, so for example if we have a Rychlik interval map, and in Theorem 6 the main application was to H\"older observables, so that the induced system is Gibbs-Markov, see for example \cite[Section 3]{Alv20}.

\section{Examples in the discrete setting}
\label{sec:disegs}
Examples of our theory require an inducing scheme and, ideally, well-understood correlation dimensions.   In \cite{PW} correlation dimension is dealt with in the Gibbs-Markov setting in the case $\{Y_i\}_i$ is a finite collection of sets, but under inducing we usually expect this collection to be infinite (in which case much less is known), so this is not directly relevant here.

The simplest case in the context of our results is when the invariant probability measure $\mu$ for the system is $d$-dimensional Lebesgue, or is absolutely continuous with respect to Lebesgue (an \emph{acip}) with a regular density, since in these cases the correlation dimension for both $\mu$ and the corresponding measure for the system is $d$. 

\subsection{Existence of the correlation dimension}
\label{ssec:cordim}
First of all, we will give a result which implies that the correlation dimension for regular acips exists.
\begin{proposition} Let $X\subset \R^d$. If $\mu$ is a probability measure on $X$ which is absolutely continuous with respect to the $d$-dimensional Lebesgue measure such that its density $\rho$ is in $L^2$, then
\[C_\mu=d.\]
\end{proposition}

\begin{proof}
The fact that $\overline{C}_\mu\le d$ follows, for example, from \cite[Theorem 1.4]{FanLauRao02}. 

To prove a lower bound, we start by defining the Hardy-Littlewood Maximal Function (see eg \cite[Chapter 2.4]{SteSha11}) of $\rho$:
$$M\rho(x) = \sup_{r>0} \frac1{Leb(B(x, r))}\int_{B(x, r)}\rho(x)~dx.$$
Moreover, by Hardy-Littlewood maximal inequality, $M\rho\in L^2$ and there exists $c_1>0$ (depending only on $d$) such that
\[\|M\rho\|_2\leq c_1 \|\rho\|_2.\]
Thus, using the Cauchy-Schwarz inequality, we have
\begin{eqnarray*}
\int \mu(B(x, r))~d\mu(x)&\le& \int M\rho(x)\cdot Leb(B(x,r))\cdot\rho(x)~dx \\
&\le& Leb(B(x,r))\left(\int\rho^2~dx\right)^{\frac12} \left(\int(M\rho)^2~dx\right)^{\frac12}\le Kr^d
\end{eqnarray*}
for some $K>0$.  
Hence $\underline{C}_\mu\ge d$ and thus $C_\mu=d$.
\end{proof}

If the density of the acip is not sufficiently regular, the correlation dimension may differ from the correlation dimension of the Lebesgue measure, as in the following case.
\begin{proposition}
Let $\alpha\in (1/2, 1)$. Assume that $\mu$ is supported on $[0, 1]$ and $d\mu= \rho dx$ with
$\rho(x) = x^{-\alpha}$.  Then, we have 
$$C_\mu = 2(1-\alpha).$$
\label{prop:nonreg}
\end{proposition}
\begin{proof}
We write $\int \mu(B(x, r))~d\mu(x) = \int_0^{2r} \mu(B(x, r))~d\mu(x)+ \int_{2r}^1 \mu(B(x, r))~d\mu(x)$.   We estimate the first term from above by
$$\int_0^{2r}\mu((0, 3r))~d\mu(x)= \mu((0, 2r))\mu((0, 3r))\asymp r^{2(1-\alpha)}.$$
For the second term
we split the sum into $\int_{nr}^{(n+1)r}x^{-\alpha}\int_{x-r}^{x+r}t^{-\alpha}~dt~dx$ for $n=2, \ldots, \lceil 1/r\rceil$.  This yields,
$$\int_{nr}^{(n+1)r}x^{-\alpha}\int_{x-r}^{x+r}t^{-\alpha}~dt~dx \le \int_{nr}^{(n+1)r} x^{-\alpha}((n-1)r)^{-\alpha}2r \lesssim r^2(nr)^{-2\alpha}.$$
Since $(1/n^{2\alpha})_n$ is a summable sequence, we estimate  $\int \mu(B(x, r))~d\mu(x)$ from above by $r^{2(1-\alpha)}$.  Therefore $\underline{C}_\mu \ge 2(1-\alpha)$.

On the other hand, since
\begin{align*}
\int_0^{r}\mu(B(x, r))~d\mu(x) & \asymp \int_0^{r}(x+r)^{1-\alpha} x^{-\alpha}~dx  \ge \int_0^{r} (x+r)^{1-2\alpha}~dx \asymp r^{2(1-\alpha)},
\end{align*}
we obtain $\overline{C}_\mu \le 2(1-\alpha)$.
\end{proof}

From now on, suppose that we are dealing with $X=\R^d$ and $\mu$, $\mu_F$ being acips with $m$ denoting normalised Lebesgue measure. We assume $\overline{\bigcup_iY_i}= Y$. First notice that if $F$ has bounded distortion (in the one-dimensional case, it  sufficient that $F$ is $C^{1+\alpha}$ with uniform constants), $\frac{d\mu_F}{dm}$ is uniformly bounded away from 0 and 1, so $C_{\mu_F} = C_{m}=d$.  

For $C_{\mu}$ we assume that $\frac{d\mu}{dm} = \rho$.  Moreover we assume there is $C>0$ with $\rho(x)\ge C$ for any $x\in X$ and that $\rho\in L^2$.   Thus $C_\mu=d$.

\subsection{Manneville-Pomeau maps}

For $\alpha \in (0,1)$, define the Manneville-Pomeau map by
\begin{equation*}
f=f_\alpha:x\mapsto \begin{cases} x(1+2^\alpha x^\alpha) & \text{ if } x\in [0, 1/2),\\
2x-1 & \text{ if } x\in [1/2, 1].\end{cases}
\end{equation*}
(This is the simpler form given by Liverani, Saussol and Vaienti, often referred to as LSV maps). This map has an acip $\mu$.  The standard procedure is to induce on $Y=[1/2, 1]$, letting $\tau$ be the first return time to $Y$.  Then by \cite[Lemma 2.3]{LivSauVai99}, $\rho\in L^2$ if $\alpha\in (0, 1/2)$.  As in for example \cite[Lemma 3.60]{Alv20}, the  map $f^\tau$ is Gibbs-Markov, so \cite[Theorem 6]{BarLiaRou19}  implies \eqref{eq:WTS}.  Thus we can apply Theorem~\ref{thm:induced_exp} to our system whenever $\alpha\in (0, 1/2)$.

In the case $\alpha\in (1/2, 1)$ then the density is similar to that in Proposition~\ref{prop:nonreg} and a similar proof gives $C_\mu=2(1-\alpha) <1 =C_{\mu_F}$, so our upper and lower bounds on the behaviour of $M_n$ do not coincide.

\subsection{Multimodal and other interval maps}

Our results apply to a wide range of interval maps with equilibrium states, for example many of those considered in \cite{DobTod20}, which guarantees the existence of inducing schemes under mild conditions.  Here we will focus on $C^3$ interval maps $f:I\to I$ (where $I=[0,1]$) with critical points with order in $(1, 2)$, i.e. for $c$ with $Df(c)=0$, there is a diffeomorphism $\phi:U\to \R$ with $U$ a neighbourhood of 0, such that if $x$ is close to $c$ then $f(x) = f(c) \pm \phi(x-c)^{\ell_c}$ for $\ell_c\in (1,2)$.  Moreover, we assume that for each critical point $c$, $|Df^n(c)|\to \infty$ and that for any open set $V\subset I$ there exists $n\in \N$ such that $f^n(V)=I$.  Then as in the main theorem in \cite{BruRivSheStr08} the system has an acip and the density is $L^2$ and hence Theorem~\ref{thm:induced_exp} applies.

\subsection{Higher dimensional examples}

We will not go into details here, but there is a large literature on nonuniformly expanding systems in higher dimensions which have acips and which have inducing schemes with tails which decay faster than polynomially.  A standard class of examples of this are the maps derived from expanding maps given in \cite{AlvBonVia00}.

\section{Orbits closeness for flows}  
\label{sec:flows}

In this section, we will extend our study to flows. First of all, as in Theorem~\ref{thm:BLR}, we will prove that an upper bound (related to the correlation dimension of the invariant measure) can be obtain in a general setting. Then, under some mixing assumptions, we will give an equivalent of Theorem~\ref{thm:induced_exp} for flows. We will prove the abstract results before giving specific examples.

Let $(X, \Psi_t, \nu)$ be a measure preserving flow on a manifold.  We will study the shortest distance between two orbits of the flow, defined by 
$$\m_t(x, y)= \m_{\Psi, t}(x, y):=\min_{0\le t_1,t_2< t} d(\Psi_{t_1}(x), \Psi_{t_2}(y)).$$

We assume that the flow has bounded speed: there exists $K\ge 0$ such that for $T>0$, $d(\Psi_t(x),\Psi_{t+T}(x))\le KT$. 

We will also assume that the flow is Lipschitz: there exists $L>0$ such that $d(\Psi_t(x),\Psi_t(y))\leq L^t d(x,y)$, and then prove an analogue of Theorem~\ref{thm:BLR}.

\begin{theorem}
For $(X, \Psi_t, \nu)$ a measure preserving Lipschitz flow with bounded speed, we have
\begin{equation*}
\limsup_{t\rightarrow+\infty}\frac{\log \m_{\Psi, t}(x, y)}{-\log t} \le \frac{2}{\underline{C}_\nu-1} \qquad \nu\times\nu\text{-a.e. } x, y.
\end{equation*}
\label{thm:limsupflow}
\end{theorem}

\begin{proof}
We define
\[S_{t,r}(x,y)=\int_0^t\int_0^t\mathbbm{1}_{B(\Psi_{t_1}(x),r)}(\Psi_{t_2}(y))dt_2 dt_1.\]
Observe that for $t>1>r$
\begin{equation*}\label{eqMnSn}
\left\{(x,y): \m_t(x,y)< r\right\}\subset \left\{(x,y):S_{2t,K_0r}(x,y)\geq r\right\}\end{equation*}
where $K_0=K+\max\{1,L\}$.

Indeed, for $(x,y)$ such that $\m_t(x,y)< r$, there exist $0\le \bar t_1,\bar t_2<t$, such that $d(\Psi_{\bar t_1}(x), \Psi_{\bar t_2}(y))<r$. Thus, for any $s\in[0,1]$ and $q\in[0,r]$, we have
\begin{eqnarray*}
d(\Psi_{\bar t_1+s}(x),\Psi_{\bar t_2+s+q}(y))&\leq&d(\Psi_{\bar t_1+s}(x),\Psi_{\bar t_2+s}(y))+d(\Psi_{\bar t_2+s}(y),\Psi_{\bar t_2+s+q}(y))\\
&\leq& L^sd(\Psi_{\bar t_1}(x),\Psi_{\bar t_2}(y))+Kq < K_0r
\end{eqnarray*}
and we obtain
\begin{eqnarray*}
S_{2t,K_0r}(x,y)&=&\int_0^{2t}\int_0^{2t}\mathbbm{1}_{B(\Psi_{t_1}(x),K_0r)}(\Psi_{t_2}(y))dt_2 dt_1\\
&\geq&\int_{0}^{1}\int_{s}^{s+r}\mathbbm{1}_{B(\Psi_{\bar t_1+s}(x),K_0r)}(\Psi_{\bar t_2 +s+q}(y))dq ds=r.
\end{eqnarray*}
Then, using Markov's inequality and the invariance of $\nu$,
\begin{align*}
\nu\otimes\nu \left((x,y):  \m_t(x,y)< r\right)&\leq \nu\otimes\nu \left((x,y): S_{2t,K_0r}(x,y)\geq r\right)\\
&\hspace{-1cm}\leq r^{-1}\E(S_{2t,K_0r})\\
&\hspace{-1cm}=r^{-1}\int_0^{2t}\int_0^{2t}\iint \mathbbm{1}_{B(\Psi_{t_1}(x),K_0r)}(\Psi_{t_2}(y))d\nu\otimes\nu(x,y)dt_2 dt_1\\
&\hspace{-1cm}=r^{-1}\int_0^{2t}\int_0^{2t}\int \nu(B(\Psi_{t_1}(x),K_0r))d\nu(x)dt_2 dt_1\\
&\hspace{-1cm}=r^{-1}(2t)^2\int\nu\left(B(x,K_0r)\right)d\nu(x).
\end{align*}
For $\eps>0$, let us define 
\[r_t=(t^{2}\log t )^{-1/(\underline{C}_\nu-1-\eps)}.\]
By the definition of the lower correlation dimension, for $t$ large enough, we have
\[\nu\otimes\nu \left((x,y):\m_t(x,y)< r_t\right)\leq (r_t)^{-1}(2t)^2(K_0r_t)^{\underline{C}_\nu-\eps}=\frac{c}{\log t}\]
with $c=4K_0^{\underline{C}_\nu-\eps}$.
Therefore, choosing a subsequence $t_\ell=\lceil e^{\ell^2}\rceil$, we have 
\[\nu\otimes\nu \left((x,y):\m_{t_\ell}(x,y)< r_{t_\ell}\right)\leq \frac{c}{\ell^2}.\]
Thus, by the Borel-Cantelli Lemma, for $\nu\otimes\nu$-almost every $(x,y)\in X\times X$, if $\ell$ is large enough then
\[\m_{t_\ell}(x,y)\geq r_{t_\ell}\]
and
\[\frac{\log \m_{t_\ell}(x,y)}{-\log t_\ell}\leq \frac{1}{\underline{C}_\nu-1-\eps}\left(2+\frac{\log\log t_\ell}{\log t_\ell}\right).\]
Finally, taking the limsup in the previous equation and observing that  $(t_\ell)_\ell$ is increasing, $(\m_t)_t$ is decreasing and $\underset{\ell\rightarrow+\infty}\lim\frac{\log t_\ell}{\log t_{\ell+1}}=1$, we have
\[ \underset{t\rightarrow+\infty}{\limsup}\frac{\log \m_t(x,y)}{-\log t}=\underset{\ell\rightarrow+\infty}{\limsup}\frac{\log \m_{t_\ell}(x,y)}{-\log t_\ell}\leq \frac{2}{\underline{C}_\nu-1-\eps}.\]
Then the theorem is proved since $\eps$ can be chosen arbitrarily small.
\end{proof}

%%%%%%%%%
%%%%%%%%%%
To obtain the lower bound, we will assume the existence of a Poincar\'e section $Y$ transverse to the direction of the flow, we denote by $\tau(x)$ the first hitting time of $x$ in $Y$, and obtain $F=\Psi_{\tau}$ on $Y$, the Poincar\'e map and $\mu$ the measure induced on $Y$.
 \begin{theorem}
  Let $(X, \Psi_t, \nu)$ a measure preserving Lipschitz flow with bounded speed. We assume that there exists a Poincar\'e section $Y$ transverse to the direction of the flow such that the Poincar\'e map $(Y,F,\mu)$, or the relevant quotiented version $(\bar Y, \bar F, \bar\mu)$, satisfies \eqref{eq:WTS}. If $C_\mu$ exists and satisfies $C_\nu=C_\mu+1$, then 
$$
\lim_{t\to +\infty}\frac{\log \m_{\Psi, t}(x, y)}{-\log t} =\frac{2}{{C}_\nu-1}=\frac{2}{{C}_\mu} \qquad \nu\times\nu\text{-a.e. } x, y.$$
\label{thm:induced-flow}
\end{theorem}

 \begin{proof}
One can mimic the proof of Theorem~\ref{thm:induced_exp}  to prove that 
$$
\liminf_{t\to +\infty}\frac{\log \m_{\Psi, t}(x, y)}{-\log t} \geq \frac{2}{C_{\mu}}=\frac{2}{C_{\nu}-1} \qquad \nu\times\nu\text{-a.e. } x, y.$$
And the result is proved using Theorem~\ref{thm:limsupflow}.
\end{proof}

We note that we are not aware of cases where $C_\nu$ and $C_\mu$ are well defined, but the condition $C_\nu=C_\mu+1$ above fails.  We give various examples in the remainder of this section of cases where these conditions hold.

\subsection{Examples of flows}
Examples where $C_\nu$ exists and there is a  Poincar\'e section as in Theorem~\ref{thm:induced-flow} with a measure $\mu$ such that  $C_\mu$ exists include Teichm\"uller flows \cite{AviGouYoc06}, a large class of geodesic flows with negative curvature, see \cite{BurMasMatWil17}, a classic example being the geodesic flow on the modular surface. In these cases, the relevant measure for (the tangent bundle on) the flow is Lebesgue, and the measure on the Poincar\'e section is an acip.

In the case of conformal Axiom A flows, the conditions of Theorem~\ref{thm:induced-flow} hold for equilibrium states of H\"older potentials, see the proof of \cite[Theorem 5.2]{PesSad01}.

\subsection{Suspension flows}\label{sec:suspflow}
 For Theorem~\ref{thm:induced-flow}, we assume that $C_\nu=C_\mu+1$. Obtaining this equality in a general setting is an open and challenging problem. In this section, we will prove that, under some natural assumptions, for suspension flows this equality holds.

Let $T:X\rightarrow X$ be a bi-Lipschitz transformation on the separable metric space $(X,d)$.  

 Let $\phi:X\rightarrow (0,+\infty)$ be a Lipschitz function. We define the space:
\[Y:=\left\{(u,s)\in X\times\R : 0\leq s\leq \phi(u)\right\}\]
where $(u,\phi(u))$ and $(Tu,0)$ are identified for all $u\in X$.
The {\it suspension flow} or the {\it special flow} over $T$ with height function $\phi$ is the flow $\Psi$ which acts on $Y$ by the following transformation
\[\Psi_t(u,s)=(u,s+t).\]
The metric on $Y$ is the Bowen-Walters distance, see \cite{BowWal72}.  Firstly, we recall the definition of the Bowen-Walters distance $d_1$ on $Y$ when $\phi(x)=1$ for every $x\in X$. Let $x,y\in X$ and $t\in [0,1]$, the length of the horizontal segment $[(x,t),(y,t)]$ is defined by:
\[\alpha_h((x,t),(y,t))=(1-t)d(x,y)+td(Tx,Ty).\]
Let $(x,t),(y,s)\in Y$ be on the same orbit, the lenght of the vertical segment $[(x,t),(y,s)]$ is defined by
\[\alpha_v((x,t),(y,s))=\inf\{|r|:\Psi_r(x,t)=(y,s)\textrm{ and }r\in\R\}.\]
Let $(x,t),(y,s)\in Y$, the distance $d_1((x,t),(y,s))$ is defined as the infimum of the lenghts of paths between $(x,t)$ and $(y,s)$ composed by a finite number of horizontal and vertical segments. 
When $\phi$ is arbitrary, the Bowen-Walters distance on $Y$ is given by
\[d_Y((x,t),(y,s))=d_1\left(\left(x,\frac{t}{\phi(x)}\right),\left(y,\frac{s}{\phi(y)}\right)\right).\]
For more details on the Bowen-Walters distance, one can see the Appendix A of \cite{BS}.

 Let $\mu$ be a $T$-invariant Borel probability measure in $X$. We recall that the measure $\nu$ on $Y$ is invariant for the flow $\Psi$ where
 \begin{equation*}\label{eqinvsusp}
 \int_Y g d\nu=\frac{\int_X\int_0^{\phi(x)}g(x,s)ds d\mu(x)}{\int_X\phi d\mu}
 \end{equation*}
for every continuous function $g:Y\rightarrow\R$. Moreover, any $\Psi$-invariant measure is of this form.  For an account of equilibrium states for suspension flows, see for example \cite{IomJorTod15}.

\begin{theorem}\label{th:dimflow}
Let $X$ be a compact space and $T:X\rightarrow X$ a bi-Lipschitz transformation. We assume that for the invariant measure $\mu$, the correlation dimension exists. If $\Psi$ is a suspension flow over $T$ as above then
\[C_\nu=C_\mu+1\]
with respect to the Bowen-Walters distance.
\end{theorem}
\begin{remark} Under the same assumptions, one can observe that if $C_\mu$ does not exist, then we have $\underline C_\nu=1+\underline C_\mu$ and $\overline C_\nu=1+\overline C_\mu$.  
\end{remark}

Before proving the theorem, we will recall some properties of the Bowen-Walters distance. First of all, for $(x,s)$ and $(y,t)\in Y$, we define
\[d_\pi((x,s),(y,t))=\min\left\{\begin{array}{l}d(x,y)+|s-t| \\ d(Tx,y)+\phi(x)-s+t \\ d(x, Ty)+\phi(y)-t+s\end{array}\right\}.\]

\begin{proposition}{\cite[Proposition 17]{BS}}
\label{propdpi}
 There exists a constant $c>1$ such that for each $(x,s)$ and $(y,t)\in Y$
\[c^{-1}d_\pi((x,s),(y,t))\leq d_Y((x,s),(y,t))\leq c d_\pi((x,s),(y,t)).\]

\end{proposition}

\begin{proof}[Proof of Theorem~\ref{th:dimflow}] We will denote $L$ a constant which is simultaneously a Lipschitz constant for $T$, $T^{-1}$ and $\phi$. 

Let $0<\eps<\frac{\min\{\phi(x)\}}{2}$. We define
\[Y_\eps=\left\{(x,s)\in Y: \eps<s<\phi(x)-\eps\right\}.\]
We will prove that for all $(x,s)\in Y_\eps$ and all $0<r<\min\{c\eps, \frac{c\eps}{L}\}$ 
\begin{enumerate}
\item $B(x,\frac{r}{2c})\times(s-\frac{r}{2c},s+\frac{r}{2c})\subset Y$ \label{bsuby}
\item $B(x,\frac{r}{2c})\times(s-\frac{r}{2c},s+\frac{r}{2c})\subset B_Y((x,s),r)$\label{bisubb}
\end{enumerate}
where $B_Y((x,s),r)$ denotes the ball centred in $(x,s)$ and of radius $r$ with respect to the distance $d_Y$.

Let $(y,t)\in B(x,\frac{r}{2c})\times(s-\frac{r}{2c},s+\frac{r}{2c})$. 

Since $s>\eps$ and $\frac{r}{c}<\eps$, we have $t>s-\frac{r}{2c}>\frac{\eps}{2}>0$.

Since $\phi$ is $L$-Lipschitz, we have $|\phi(x)-\phi(y)|\leq L d(x,y)<\frac{Lr}{2c}$. Moreover, since $s<\phi(x)-\eps$, we obtain
\begin{eqnarray*}
t&<&s+\frac{r}{2c} <\phi(x)-\eps+\frac{\eps}{2}\\
&< &\phi(y)+\frac{Lr}{2c}-\frac{\eps}{2}< \phi(y).
\end{eqnarray*}
Thus $(y,t)\in Y$ and \eqref{bsuby} is proved.

For $(y,t)\in B(x,\frac{r}{2c})\times(s-\frac{r}{2c},s+\frac{r}{2c})$, we can use Proposition~\ref{propdpi} to obtain
\begin{eqnarray*}
d_Y((x,s),(y,t))&\leq& c d_\pi((x,s),(y,t))\\
&\leq & c\left(d(x,y)+|s-t|\right)\\
&<&c\left(\frac{r}{2c}+\frac{r}{2c}\right)=r
\end{eqnarray*}
and 
\eqref{bisubb} is proved.

We can now use \eqref{bsuby} and \eqref{bisubb} to obtain an upper bound for $C_\nu$. For $0<r<\min\{c\eps, \frac{c\eps}{L}\}$, we have
\begin{align*}
\int_Y\nu(B_Y((x,s),r)) & d\nu(x,s)\geq  \int_Y\mathbbm{1}_{Y_\eps}(x,s)\nu(B_Y((x,s),r))d\nu(x,s)\\
&\geq\frac{1}{\int_X\phi d\mu}\int_X\int_\eps^{\phi(x)-\eps}\nu(B_Y((x,s),r))ds d\mu(x)\\
&\geq \frac{1}{\int_X\phi d\mu}\int_X\int_\eps^{\phi(x)-\eps}\nu\left(B\left(x,\frac{r}{2c}\right)\times\left(s-\frac{r}{2c},s+\frac{r}{2c}\right)\right)ds d\mu(x)\\
&= \left(\frac{1}{\int_X\phi d\mu}\right)^2\int_X\int_\eps^{\phi(x)-\eps}\frac{r}{c}\mu\left(B\left(x,\frac{r}{2c}\right)\right)ds d\mu(x)\\
&\geq\left(\frac{1}{\int_X\phi d\mu}\right)^2\min(\phi(x)-2\eps)\frac{r}{c}\int_X\mu\left(B\left(x,\frac{r}{2c}\right)\right)d\mu(x)\\
&\geq C_1 r\int_X\mu\left(B\left(x,\frac{r}{2c}\right)\right)d\mu(x)
\end{align*}
with $C_1=\left(\frac{1}{\int_X\phi d\mu}\right)^2\min(\phi(x)-2\eps)\frac{1}{c}>0$. We conclude that
\begin{equation}\label{cnuleqcmu}
\limsup_{r\rightarrow0}\frac{\log\int_Y\nu(B_Y((x,s),r))d\nu(x,s)}{\log r}\leq \lim_{r\rightarrow0}\frac{\log C_1 r\int_X\mu(B(x,\frac{r}{2c}))d\mu(x)}{\log r}=1+C_\mu.
\end{equation}
To prove the lower bound, we define, for $(x,s)\in Y$, the sets
\begin{eqnarray*}
B_1&=& B(x,cr)\times(s-rc,s+rc)\\
B_2 &=& B(Tx,cr)\times[0,rc)\\
B_3 &=& \{(y,t)\in Y:y\in B(T^{-1}x,Lrc)\textrm{ and }\phi(y)-rc<t\leq \phi(y)\}.
\end{eqnarray*}
We have 
\[B_Y((x,s),r)\subset (B_1\cup B_2\cup B_3)\cap Y.\]
Indeed, if $(y,t)\in B_Y((x,s),r)$, then, using Proposition~\ref{propdpi}, we have $d_\pi((x,s),(y,t))\leq c d_Y((x,s),(y,t))<cr$. Thus, by definition of $d_\pi$, there are three possibilities:
\begin{itemize}
\item  If $d_\pi((x,s),(y,t))=d(x,y)+|s-t|$, then $d(x,y)<cr$ and $|s-t|<cr$. Thus $(y,t)\in B_1$;
\item If $d_\pi((x,s),(y,t))=d(Tx,y)+\phi(x)-s+t$, then $d(Tx,y)<cr$ and $0\leq t<cr$ (since $\phi(x)-s\geq 0$ and $(y,t)\in Y$). Thus $(y,t)\in B_2$;
\item If $d_\pi((x,s),(y,t))=d(x,Ty)+\phi(y)-t+s$, then $d(T^{-1}x,y)\leq L d(x,Ty)<Lcr$. Since $s\geq0$, we have $\psi(y)-t<cr$ and since $(y,t)\in Y$, we have $t\leq \phi(y)$ . Thus $(y,t)\in B_3$.
\end{itemize}
Using the definition of $\nu$ we have:
\begin{eqnarray*}
\nu(B_1\cap Y)&\leq&\frac{1}{\int_X\phi d\mu}2rc\mu(B(x,cr)),\\
\nu(B_2\cap Y)&\leq&\frac{1}{\int_X\phi d\mu}rc\mu(B(Tx,cr)),\\
\nu(B_2\cap Y)&\leq&\frac{1}{\int_X\phi d\mu}rc\mu(B(T^{-1}x,Lcr)).
\end{eqnarray*}

Denoting $c_1=\max\{c,Lc\}$, we have
\begin{align*}
\int_Y\nu(B_Y((x,s),r))d\nu(x,s)&\leq \int_Y\nu(B_1\cap Y)+\nu(B_2\cap Y)+\nu(B_3\cap Y)d\nu(x,s)\\
&\hspace{-3.5cm}\leq  \frac{2c}{\int_X\phi d\mu}r\left(\int_X \mu(B(x,c_1r))d\mu+\int_X \mu(B(Tx,c_1r))d\mu+\int_X \mu(B(T^{-1}x,c_1r))d\mu\right)\\
&\hspace{-3.5cm}=\frac{6c}{\int_X\phi d\mu}r\int_X \mu(B(x,c_1r))d\mu
\end{align*}
since $\mu$ is $T$-invariant and $T^{-1}$-invariant. 

Finally we obtain
\begin{equation}\label{cnugeqcmu}
\liminf_{r\rightarrow0}\frac{\log\int_Y\nu(B_Y((x,s),r))d\nu(x,s)}{\log r}\geq \lim_{r\rightarrow0}\frac{\log \frac{6c}{\int_X\phi d\mu}r\int_X \mu(B(x,c_1r))d\mu}{\log r}=1+C_\mu.
\end{equation}
Thus by \eqref{cnuleqcmu} and \eqref{cnugeqcmu}, the theorem is proved.
\end{proof}

\section{A class of examples with orbits remoteness}
\label{sec:skew}

In this section, we give an example of a class of mixing systems were \eqref{eq:WTS} fails to hold, see Remark~\ref{rmk:noind} below for the relation to the other results in this paper. This family of systems was defined in \cite{GRS} and its mixing and recurrence/hitting times properties were studied.

We will consider a class of systems constructed as follows. The base is a
measure preserving system $(\Omega ,T,\mu )$. We assume that $T$ is a
piecewise expanding Markov map on a finite-dimensional Riemannian manifold $%
\Omega $, that is:

\begin{itemize}
\item there exists some constant $\beta >1$ such that $\Vert D_{x}T^{-1}\Vert
\leq \beta^{-1}$ for every $x\in \Omega $.

\item There exists a collection $\mathcal{J}=\{J_1,\ldots,J_p\}$ such that
each $J_i$ is a closed proper set and

(M1) $T$ is a $C^{1+\eta}$ diffeomorphism from $\inte J_i$ onto its image;

(M2) $\Omega =\cup_i J_i$ and $\inte J_i\cap \inte J_j=\emptyset$ unless $%
i=j $;

(M3) $T(J_i)\supset J_j$ whenever $T(\inte J_i)\cap \inte J_j\neq\emptyset$.
\end{itemize}

$\mathcal{J}$ is called a Markov partition. It is well known that such a
Markov map is semi-conjugated to a subshift of finite type. Without loss of
generality we assume that $T$ is topologically mixing, or equivalently that
for each $i$ there exists $n_i$ such that $T^{n_i}J_i=\Omega$. We assume
that $\mu$ is the equilibrium state of some potential $\psi\colon\Omega\to 
\mathbb{R}$,  H\"older continuous in each interior of
the $J_i$'s. 
The sets of the form $J_{i_0,\ldots,i_{q-1}}:=\bigcap_{n=0}^{q-1} T^{-n}J_{i_n}$ are called \emph{cylinders of size $q$} and we
denote their collection by $\mathcal{J}_q$.

In this setting, the correlation dimension of $\mu$ exists as in \cite[Theorem 1]{PW}.  Note that we could arrange our system so that our $\mu$ an acip: the density here will be bounded, so the correlation dimension is one.

The system is extended by a skew product to a system $(M,S)$ where $M=\Omega
\times \mathbb{T}$ and $S:M\rightarrow M$ is defined by
\begin{equation*}
S(\omega ,t)=(T\omega ,t+\alpha\varphi (\omega ))  \label{skewprod}
\end{equation*}
where $\varphi =1_{I}$ is the characteristic function of a set $I\subset \Omega $
which is a union of cylinders. In this system the second coordinate is
translated by $\mathbf{\alpha }$ if the first coordinate belongs to $I$. We
endow $(M,S)$ with the invariant measure $\nu=\mu \times Leb$ (so $C_\nu=C_\mu+1$). On $\Omega \times \mathbb{T}$ we will consider the sup distance.

We make the standing assumption on our choice of $\varphi$ that 

\begin{itemize}
\item (NA) for any $u\in \lbrack -\pi ,\pi ]$, the
equation $fe^{iu\varphi }=\lambda f\circ T$, where $f$ is H\"{o}lder (on the subshift) and $\lambda
\in S^{1}$, has only the trivial solutions $\lambda =1$ and $f$ constant. 
\end{itemize}

The
simple case where the $I$ which defines $\varphi$ is a nonempty union of size $1$ cylinders such
that both $I$ and $I^{c}$ contain a fixed point fulfils this assumption.

\begin{definition}
\label{type}Given an irrational number $\alpha $ we define the irrationality exponent
of $\alpha $ as the following (possibly infinite) number: 
\begin{equation*}\label{gamma}
\gamma (\alpha )=\inf \{\beta :\liminf_{q\rightarrow \infty }q^{\beta }\|
q\alpha \| >0\}
\end{equation*}
where $\| \cdot \|$ indicates the distance to the
nearest integer number in $\mathbb{R}$.
\end{definition}

First note that $\gamma(\alpha)\ge 1$ for any irrational $\alpha$. 

\begin{remark} By \cite[Theorem 19]{GRS}, if $\gamma(\alpha)>d_\mu+1$, then the Hitting Time Statistics is typically degenerate.  This is an indirect way of seeing that there cannot be an inducing scheme satisfying \eqref{eq:WTS}, otherwise \cite[Theorem 2.1]{BruSauTroVai03} would be violated; it also suggests that the conclusions of Theorem~\ref{thm:induced_exp} will also not hold here, which we show below is indeed the case.
\label{rmk:noind}
\end{remark}

\begin{theorem}
For $\nu\times\nu$-a.e. $x,y\in M$ we have
\begin{equation*}
\limsup_{n}\frac{\log \m_{S, n}(x, y)}{-\log n}\leq\min\left(\frac{2}{C_\nu},1\right)=\min\left(\frac{2}{C_\mu+1},1\right)
\end{equation*}
and
\begin{equation}\label{eqliminfgamma}
\liminf_{n}\frac{\log \m_{S,n}(x, y)}{-\log n}\leq\min\left(\frac{2}{C_\nu},\frac{1}{\gamma(\alpha)}\right)=\min\left(\frac{2}{C_\mu+1},\frac{1}{\gamma(\alpha)}\right).
\end{equation}
\end{theorem}

\begin{proof}
First of all, applying Theorem~\ref{thm:BLR} to $S$ and since one can easily show that $C_\nu=C_\mu+1$, we obtain for $\nu\times\nu$-a.e. $x,y\in M$

\begin{equation*}
\liminf_{n}\frac{\log \m_{S,n}(x, y)}{-\log n}\leq\limsup_{n}\frac{\log \m_{S, n}(x, y)}{-\log n}\leq\frac{2}{C_\nu}=\frac{2}{C_\mu+1}.
\end{equation*}

Moreover, one can observe that for $x=(\omega,t)\in M$ and $y=(\tilde\omega,s)\in M$ 
\begin{align*}
\m_{S,n}(x,y)&= \min_{0\le i,j\le n-1} \max\left(d(T^i(\omega), T^j(\tilde\omega)),\|(t-s)+\alpha(S_i\varphi(\omega)-S_j\varphi(\tilde\omega))\|\right)\\
&\hspace{-1cm}\geq \max\left(\min_{0\le i,j\le n-1} d(T^i(\omega), T^j(\tilde\omega)),\min_{0\le i,j\le n-1}\|(t-s)+\alpha(S_i\varphi(\omega)-S_j\varphi(\tilde\omega))\|\right)\\
&\hspace{-1cm}\geq \max\left(\min_{0\le i,j\le n-1} d(T^i(\omega), T^j(\tilde\omega)),\min_{-(n-1)\le i\le n-1}\|(t-s)+i\alpha)\|\right)\\
&\hspace{-1cm}=\max \left(\m_{T,n}(\omega,\tilde\omega),\m_{R,n}(t,s)\right)
\end{align*}
where $R:\mathbb{T}\mapsto\mathbb{T}$ with $R(s)=s+\alpha$.
Thus, by \cite[Theorems 1 and 10]{BarLiaRou19} we obtain
\begin{eqnarray*}
\limsup_n\frac{\log \m_{S, n}(x, y)}{-\log n}&\leq&\min\left(\limsup_n\frac{\log \m_{T,n}(\omega,\tilde\omega)}{-\log n},\limsup_n\frac{\log\m_{R,n}(t,s)}{-\log n}\right)\\
&\leq&\min\left(\frac{2}{C_\mu},1\right)
\end{eqnarray*}
and
\begin{eqnarray*}
\underset{n}{\liminf}\frac{\log \m_{S,n}(x, y)}{-\log n}&\leq&\min\left(\underset{n}{\liminf}\frac{\log \m_{T,n}(\omega,\tilde\omega)}{-\log n},\underset{n}{\liminf}\frac{\log\m_{R,n}(t,s)}{-\log n}\right)\\
&\leq&\min\left(\frac{2}{C_\mu},\frac{1}{\gamma(\alpha)}\right).
\end{eqnarray*}
Finally, since $C_\nu=C_\mu+1>C_\mu$, the theorem is proved.
\end{proof}

Finally, we prove that if $\mu$ is a Bernoulli measure, then inequality \eqref{eqliminfgamma} is sharp.

\begin{theorem}
We assume that all the branches of the Markov map $T$ are full, i.e. $%
T(J_{i})=\Omega $ for all $i$, that $\mu $ is a Bernoulli measure i.e. $\mu
([a_{1}\ldots a_{n}])=\mu ([a_{1}])\cdots \mu ([a_{n}])$, and $I$ depends
only on the first symbol, i.e. $I$ is an union of $1$-cylinders (recall that 
$\varphi =1_{I}$).

If $\gamma(\alpha)> d_\mu+1$ then
\begin{equation}
\liminf_{n}\frac{\log \m_{S,n}(x, y)}{-\log n}=\frac{1}{\gamma(\alpha)}< \frac{2}{C_\nu} \qquad \nu\times\nu\text{-a.e. } x,y,
\end{equation}
\end{theorem}

\begin{proof}
First of all, we recall that $C_\mu\leq d_\mu$ (see e.g. \cite{Pes93}), thus our assumption on $\alpha$ implies that $1/\gamma(\alpha)<2/C_\nu$, so \eqref{eqliminfgamma} implies $$\liminf_{n}\frac{\log \m_{S,n}(x, y)}{-\log n}\le\frac{1}{\gamma(\alpha)}.$$  So it remains to show the reverse of the above inequality.

By \cite[Proposition 21]{GRS}, for any $y$, for $\nu$-a.e. $x$ we have 
\begin{equation}\label{waiting}
\limsup_{r\rightarrow0}\frac{\log W_r(x,y)}{-\log r}\leq \max ({d}_{\mu }+1,\gamma(\alpha))
\end{equation}
where $ W_r(x,y)=\inf\{k\geq1, S^k(x)\in B(y,r)\}$.

Let $\epsilon>0$ and let $x,y$ such that \eqref{waiting} holds. Since $\gamma(\alpha) \geq d_\mu+1$, for any $r$ small enough we have
\[ W_r(x,y)\leq r^{-(\gamma(\alpha)+\epsilon)}\]
which implies that
\[\m_{S, \lceil r^{-(\gamma(\alpha)+\epsilon)}\rceil}(x,y)<r.\]
Thus, for any $r$ small enough
\[\frac{\log \m_{S, \lceil r^{-(\gamma(\alpha)+\epsilon)} \rceil}(x,y)}{-\log  \lceil r^{-(\gamma(\alpha)+\epsilon)} \rceil}>\frac{1}{\gamma(\alpha)+\epsilon}\]
and then 
\[\liminf_{n}\frac{\log \m_{S,n}(x, y)}{-\log n}\geq\frac{1}{\gamma(\alpha)+\epsilon}.\]
The theorem is proved taking $\epsilon$ arbitrary small.
\end{proof}

\end{document}